\documentclass[12pt,reqno]{amsart}
\usepackage[utf8]{inputenc}
\usepackage{amsmath}
\usepackage{calc}
\usepackage{color}
\usepackage{cleveref}
\usepackage{lipsum}
\usepackage{amssymb}
\usepackage{graphicx}
\usepackage{epstopdf}

\usepackage{accents}
\newcommand{\dbtilde}[1]{\accentset{\approx}{#1}}

\newtheorem{theorem}{Theorem}[section]

\newtheorem{prop}[theorem]{Proposition}

\crefname{proposition}{Proposition}{lemmas}
\theoremstyle{definition}

\numberwithin{equation}{section}

\date{}

\begin{document}
\baselineskip=15.5pt

\title{Semisimple algebras of vector fields on $\mathbb{C}^{3}$}

\author[S. Ali]{Sajid Ali$^{1}$}

\address{$^{1}$College of Interdisciplinary Studies, Zayed University, Dubai 19282, UAE}

\email{sajid.ali@zu.ac.ae}

\author[H. Azad]{Hassan Azad$^{2}$}

\address{$^{2}$Abdus Salam School of Mathematical Sciences, GCU, Lahore 54600, Pakistan}

\email{hssnazad@gmail.com}

\author[I. Biswas]{Indranil Biswas$^{3}$}

\address{$^{3}$Department of Mathematics, Shiv Nadar University, NH91, Tehsil
Dadri, Greater Noida, Uttar Pradesh 201314, India}

\email{indranil.biswas@snu.edu.in, indranil29@gmail.com}

\author[F. M. Mahomed]{Fazal M. Mahomed$^{4}$}

\address{$^{4}$DSI-NRF Centre of Excellence in Mathematical and Statistical Sciences, School
of Computer Science and Applied Mathematics, University of the Witwatersrand,
Johannesburg, Wits 2050, South Africa}

\email{Fazal.Mahomed@wits.edu.za}

\author[S. W. Shah]{Said Waqas Shah$^{2}$}

\address{$^{2}$Abdus Salam School of Mathematical Sciences, GCU, Lahore 54600, Pakistan}

\email{waqas.shah@sms.edu.pk}

\begin{abstract}
A local classification of semisimple algebras of vector fields on $\mathbb{C}^{3}$ is given,
using the canonical forms of the Heisenberg algebra and of $sl(2,\mathbb{C})\times sl(2,\mathbb{C})$.
\end{abstract}

\maketitle

\section{Introduction}

Lie classified all finite dimensional algebras of vector fields in the plane up to local
equivalence (see \cite{TTG}). He also attempted a similar classification of vector fields on $\mathbb{C}^3$ but details are
missing. We refer the reader to \cite{Olver} for a detailed history of this problem.

It seems that this problem is intractable as even the nilpotent algebras have no satisfactory classification.
However, semisimple algebras --- and nonsolvable algebras --- do have a satisfactory answer. This problem is
closely related to the classification of 3-d homogeneous spaces. In \cite{Dubrov}, such a classification is
given. In this paper, we do not assume that the Lie algebra is transitive or that the vector fields are
complete. A sketch of the argument is given in Section 2 of this paper. All Lie algebras considered in this
paper are finite dimensional and the classification is up to local equivalence. The main result of this paper are
given in Proposition \ref{prop1}, Proposition \ref{prop2}, Proposition \ref{prop3}, Proposition
\ref{prop4} and Proposition \ref{prop5}.

\section{Sketch of the Main Argument}

The semisimple algebras of vector fields on $\mathbb{C}^{N}$ that have a Cartan algebra of dimension $N$ have
been classified in \cite{ABM1}. Their simple factors can only be of type $A$ and their canonical forms are
given in \cite{ABM2}. Moreover, semisimple algebras of vector fields on $\mathbb{C}^{N}$ have Cartan algebras of
dimension at most $N$. Thus, to describe all semisimple algebras of vector fields on $\mathbb{C}^3$, we need
only consider the algebras of rank at most $2$ (as the case
where rank is three has already been covered). Algebras of rank $2$ are of types $A_{1} \times A_{1}$, $A_{2}$,
$B_{2}$ or $G_{2}$.

Let $V(\mathbb{C}^3)$ denote the Lie algebra of vector fields on $\mathbb{C}^3$.

The Lie algebra $A_2$ contains the 3-d Heisenberg algebra: $[X,\,Y]\,=\,Z$. $B_{2}$ contains
$A_{1} \times A_{1}$, while $G_{2}$ contains an algebra of type $A_2$. However, any type $A_{2}$ in
$V(\mathbb{C}^3)$ cannot be extended to $G_{2}$. The canonical forms of the Heisenberg algebra and
of $A_{1} \times A_{1}$, together with the commutation relations given by the root systems,
are then used in determining the fundamental root vectors to find all the embeddings of rank
two algebras in $V(\mathbb{C}^3)$.

\section{Canonical Forms of the Heisenberg Algebra}

The 3-d Heisenberg algebra is the algebra with basis $X,\,Y$ and $Z$ with $[X,\,Y] \,=\, Z$ and $Z$ in the center of
the algebra. In \cite{Wafo} and \cite{Boyko} the canonical forms of low dimensional algebras
of vector fields are given. We give a derivation of the canonical forms of the Heisenberg algebra because
similar ideas are used in giving the canonical forms of higher dimensional algebras.

\begin{prop}\label{prop1}
If $\mathfrak g$ is a 3-d subalgebra of $V(\mathbb{C}^3)$ with basis $\langle X,\,Y,\,Z\rangle$ such that $[X,\,Y]\,=\,Z$ and $Z$ is in the
center of $\mathfrak g$, then $\mathfrak g$ is equivalent to one of the following forms:
\begin{enumerate}
\item $Z\,=\, \partial_{x}$,\, $X \,=\, \partial_{y}$ and $Y\,=\, y \partial_{x} + \partial_z$;

\item $Z\,=\, \partial_{x}$,\, $X \,=\, \partial_{y}$ and $Y \,=\, y \partial_{x} + \lambda \partial_{y}$;

\item $Z\,=\, \partial_{x}$,\, $X \,=\, \partial_{y}$ and $Y \,=\, y \partial_{x} + z\partial_{y}$.
\end{enumerate}
\end{prop}

\begin{proof}
Choose (local) coordinates $x,\, y,\, z$ in which $Z\,=\, \partial_{x}$. In these coordinates
\begin{align}
&X \,\,=\,\, a_{1} (y,\,z) \partial_{x} + b_1 (y,\,z) \partial_y + c_{1} (y,\,z) \partial_{z},\nonumber \\
&Y \,\,=\,\, a_{2} (y,\,z) \partial_{x} + b_2 (y,\,z) \partial_y + c_{2} (y,\,z) \partial_{z}.
\end{align}
Now one of $\langle X,\,Z\rangle$ and $\langle Y,\,Z\rangle$ must be of rank two, otherwise
$[X,\,Y]\,=\,0$. Say ${\rm rank}(\langle X,\, Z\rangle)\,=\, 2$. We can introduce local coordinates so
that $X\,=\,\partial_{y}$ and $Z\,=\,\partial_{x}$. In these coordinates, $Y\,=\, a (y,\,z) \partial_{x} + b (y,\,z) \partial_y +
c(y,\,z) \partial_{z}$. Therefore $[X,\,Y]\,=\,Z$ implies that
\begin{align}
\frac{\partial a}{\partial y} \,=\, 1 ,\ \frac{\partial b}{\partial y} \,=\, 0,\ \frac{\partial c}{\partial y} \,=\, 0.
\end{align}
Hence, $Y\,= y\, \partial_x + V$ where $V\,=\, \phi_{1}(z)\partial_{x} +\phi_{2}(z)\partial_{y} +\phi_{3}(z)\partial_{z}$.

\textbf{Case I:}\, {\it The rank of $\langle \partial_x,\, \partial_y,\, V\rangle$ is three.}\, Therefore we can introduce coordinates
$\widetilde{x},\, \widetilde{y},\, \widetilde{z}$ so that $\partial_x \,=\, \partial_{\widetilde{x}}$,
$\partial_y \,=\, \partial_{\widetilde{y}}$ and $v \,=\, \partial_{\widetilde{z}}$. Thus, $\widetilde{x}
\,=\, x+F(z)$, $\widetilde{y}\,=\, y + G(z)$ and $\widetilde{z} \,=\, H(z)$. In these coordinates,
\begin{equation}
Z\,=\, \partial_{\widetilde{x}}\,=\, \partial_x,\ X \,=\, \partial_{\widetilde{y}}\,=\, \partial_y,\ Y
\,=\, y\partial_x +\partial_{\widetilde z} \,=\, (\widetilde{y}- G(\widetilde{z}))\partial_{\widetilde{x}} + \partial_{\widetilde{z}}.
\end{equation}
Thus by change of notation, we may assume that
\begin{equation}
Z\,=\, \partial_{{x}},\ X \,=\, \partial_{{y}},\ Y \,=\, ({y}+f({z}))\partial_{x} + \partial_{{z}}.
\end{equation}
By the change of variables $$\widetilde{x} \,=\, x- \int f(z) dz ,\ \,  \widetilde{y} \,=\, y,\ \, \widetilde{z} \,=\, z,$$ we have
\begin{equation}
\partial_{x} \,=\, \partial_{\widetilde{x}},\ \partial_{y}\,=\, \partial_{\widetilde{y}},\ \partial_{z}
\,=\, -f(z)\partial_{\widetilde{x}} + \partial_{\widetilde{z}}.
\end{equation}
Therefore, $Y\,=\, (y+f(z))\partial_x+\partial_z\,=\,(y+f(z))\partial_{\widetilde x}- f(z)\partial_{\widetilde x}+\partial_{\widetilde z}
\,=\, y \partial_{\widetilde{x}} + \partial_{\widetilde{z}}\,=\, \widetilde{y} \partial_{\widetilde{x}} + \partial_{\widetilde{z}}$.
Thus in these coordinates,
\begin{equation}
Z \,=\, \partial_{\widetilde{x}},\ X \,=\, \partial_{\widetilde{y}},\ Y\,=\, \widetilde{y}\partial_{\widetilde{x}} +
\partial_{\widetilde{z}}.
\end{equation}

\textbf{Case II:} {\it The rank of $\langle \partial_x,\, \partial_y,\, V\rangle$ is two, where $Z\,=\, \partial_x$, $X\,=\,
\partial_y$ and $Y\,=\, y\partial_x+V$ with
$V\,= \phi_{1}(z)\partial_{x} +\phi_{2}(z)\partial_{y} +\phi_{3}(z)\partial_{z}$.}\, This implies that $\phi_{3}(z) \,=\,0$ and
$Y\,=\, (y + \phi_{1}(z))\partial_{x} + \phi_{2} (z) \partial_{y}$. By change of coordinates $\widetilde{x} \,=\, x$,
\, $\widetilde{y}\,=\, y + \phi_{1}(z)$ and $\widetilde{z} \,=\,z$, we have
$$
\partial_{x} \,=\, \partial_{\widetilde{x}},\
\partial_{y}  \,=\, \partial_{\widetilde{y}},\  Y \,=\, \widetilde{y} \partial_{\widetilde{x}} +
\phi_{2} (\widetilde{z}) \partial_{\widetilde{y}}.
$$
Hence, we may assume that
\begin{equation}\label{ep}
Z\,=\, \partial_{x} ,\,\ X \,=\, \partial_{y},\,\ Y \,=\, y \partial_{x}+ \phi(z) \partial_{y}.
\end{equation}

\textbf{Subcase IIa:}\, {\it $\phi(z)$ in \eqref{ep} is a constant.}\, In this case, the
algebra $\mathfrak g$ is $\langle \partial_{x},\, \partial_{y},\, y \partial_{x}\rangle$ and
its centralizer is $A(z) \partial_{x} + B(z) \partial_{z}$ and the centralizer is of rank 2.

\textbf{Subcase IIb:} {\it $\phi(z)$ is not a constant.}\, In this case by the change of
variables $\widetilde{x} \,=\, x$,\, $\widetilde{y} \,=\, y$ and $\widetilde{z} \,=\,
\phi(z)$, the algebra is
\begin{equation}
\langle \partial_{\widetilde{x}},\ \partial_{\widetilde{y}},\ \widetilde{y} \partial_{\widetilde{x}} +
\widetilde{z} {\partial_{y}}\rangle
\end{equation}
and its centralizer is $A(\widetilde{z}) \partial_{\widetilde{x}}$ and it is of rank 1. As the algebra $\langle\partial_{{x}},\, \partial_{{y}},\, {y} \partial_{{x}} + \partial_{z}\rangle$ is
of rank 3 whereas algebras
$\langle\partial_{{x}},\, \partial_{{y}},\, {y} \partial_{{x}} + \lambda {\partial_{y}}\rangle$ and
$\langle\partial_{{x}},\, \partial_{{y}},\, {y} \partial_{{x}} + {z} {\partial_{y}}\rangle$ are of rank 2, we see that
these three forms of Heisenberg algebra are inequivalent.
\end{proof}

\section{Canonical Forms of $sl(2,\mathbb{C})$ in $V(\mathbb{C}^{3})$}

Take a basis $X,\,Y,\, H$ of $sl(2,\mathbb{C})$ with $[H,\,X] \,=\, X,\ [H,\,Y] \,=\, -Y$ and $[X, \, Y] =\, H$. Choose local coordinates $x,\,y,\,z$ in a neighborhood
of a point where $H$ is not zero so that $H \,=\, \partial_{x}$. Thus in these coordinates
\begin{align}
&X \,\,=\,\, e^{x} ( a (y,z) \partial_{x} + b (y,z) \partial_y + c (y,z) \partial_{z}),\nonumber \\
&Y \,\,=\,\, e^{-x} (a_{1} (y,z) \partial_{x} + b_1 (y,z) \partial_y + c_{1} (y,z) \partial_{z}).
\end{align}

\textbf{Case I:}\, Assume that the rank of $\mathfrak g$ is one and we have chosen the coordinates in which $H\,=\,\partial_{x}$.\,
In this case, $X\,=\,e^{x}a(y,z)\partial_{x}$,\, $Y\,=\, e^{-x} a_{1}(y,z)\partial_x$. If $a(y,z)$ is constant then
$[X,\,Y]\,=\,\partial_{x}$ implies that $a_{1}$ is also a constant. Thus,
\begin{equation}
{\mathfrak g}\,\, =\,\, \langle e^{x}\partial_{x},\ e^{-x} \partial_{x},\ \partial_{x}\rangle.
\end{equation}
Now suppose $a$ is not a constant, say $\frac{\partial a}{\partial y} \,\neq\, 0$. By the change of variables $\widetilde{x} \,=\, x$,\,
$\widetilde{y} \,=\, a(y,z)$,\, $\widetilde{z}\,=\, z$,
\begin{equation}
\partial_x\,=\, \partial_{\widetilde x},\ \partial_y\,=\, \frac{\partial a}{\partial y}\partial_{\widetilde y},\
\partial_z\,=\, \frac{\partial a}{\partial z}\partial_{\widetilde y}+\partial_{\widetilde z}.
\end{equation}
Thus
\begin{equation}
X \,= \,e^{\widetilde{x}} \widetilde{y} \partial_{\widetilde{x}},\ Y\,=\, e^{-\widetilde{x}} {a}_1(\widetilde{y},
\,\widetilde{z}) \partial_{\widetilde{x}}.
\end{equation}
By change of notation,
\begin{equation}
X \,=\, e^{{x}} {y} \partial_{{x}},\ Y\,=\, e^{-{x}} {a}_1({y}, {z}) \partial_{{x}}.
\end{equation}
Thus
\begin{equation}
X\,\,=\,\, e^{x+\ln(y)}\partial_x,\ \, Y\,\,=\,\, e^{-x}a_1(y,z)\partial_x.
\end{equation}
If we put $\widetilde{x} \,=\, x+\ln {y},\ \widetilde{y} \,=\, y,\ \widetilde{z} \,=\, z$, then
\begin{equation}
X \,= \, e^{\widetilde{x}}  \partial_{\widetilde{x}},\  Y\,\,= \,\,e^{-(\widetilde{x}-\ln y)}\partial_{\widetilde x}\,=\,
e^{-\widetilde x}\widetilde{y}a_1(\widetilde{y},\widetilde{z})\partial_{\widetilde x}\,=\,
e^{-\widetilde x}a(\widetilde{y},\widetilde{z})\partial_{\widetilde x}.
\end{equation}
As $[X,\,Y] \,=\, H \,=\, \partial_{\widetilde{x}}$, we must have $a(\widetilde{y}, \widetilde{z})$ as constant. Therefore, by a
change of variables
\begin{equation}
{\mathfrak g}\,\, =\,\, \langle e^{x}\partial_{x},\ e^{-x} \partial_{x},\ \partial_{x}\rangle .
\end{equation}

\textbf{Case II:}\, Assume that the rank of $\mathfrak g$ is two.\, Say rank of $\langle X,\,H\rangle$ is two. Thus
we may assume that
\begin{align}
&H\,=\,\partial_{x},\, X\,=\,e^{x}(a(y,z)\partial_{x}+b(y,z)\partial_{y}),\, Y\,=\, e^{-x}(a_1(y,z)\partial_{x}+
b_1(y,z)\partial_{y})
\end{align}
and $b\,\neq\, 0$. Thus there is a change of variables $\widetilde{x},\, \widetilde{y},\, \widetilde{z}$ in which
$\partial_{x}\,=\, \partial_{\widetilde{x}}$, $a(y,z)\partial_{x} + b(y,z)\partial_{y} \,=\, \partial_{\widetilde{y}}$. Hence,
$\widetilde{x} \,= \,x+\phi(y,z)$ and $\widetilde{y},\, \widetilde{z}$ are functions of $y$ and $z$. Thus, in these coordinates
\begin{equation}
X \,=\, e^{\widetilde{x}-\widetilde{\phi}(\widetilde{y}, \widetilde{z})}\partial_{\widetilde{y}},\,\  H\,=\, \partial_{\widetilde{x}}.
\end{equation}
By the change of variables $\dbtilde{x} \,=\, \widetilde{x}$,\, $e^{-\widetilde{\phi}(\widetilde{y},\widetilde{z})} \partial_{\widetilde{y}}
\,=\, \partial_{\dbtilde{y}}$,\, $\widetilde{z} \,= \,\dbtilde{z}$ and by a change of notation, we may assume that
\begin{equation}\label{e2}
X \,=\, e^{{x}} \partial_{{y}},\ H\,=\, \partial_{x},\ Y\,=\, e^{-{x}} ({a}({y}, {z}) \partial_{{x}} +b(y,z) \partial_{y}),
\end{equation}
and $b(y,z)\,\neq\, 0.$

\textbf{Case IIa:}\, \textit{$b(y,\,z)$ in \eqref{e2} is a non-zero constant $k$.}\,  In this case,
$$
\partial_{x} \,\,=\,\, [ e^{x} \partial_{y} ,\ e^{-x}(a\partial_{x} + k\partial_{y})]
$$
gives $a\,=\,0$ and $\frac{\partial a}{\partial y}\,=\, 1$. Thus, this is not possible.

\textbf{Case IIb:}\, \textit{ Assume that $b_{z}\,=\,0$.}\, In this case
$$
H\,=\, \partial_{x},\ X\,=\, e^{x} \partial_{y},\ Y \,=\,
e^{-x}(a(y,z)\partial_{x} + b(y) \partial_{y}).
$$
Hence $[X,\,Y]\,=\, \partial_{x}$ gives $a\,=\,y + \phi(z)$ and $b'(y) \,=\, y+\phi(z).$ Therefore, $\phi(z)$ is a constant, say $k$ and
$a\,=\,y+k$, $b\,=\, \frac{(y+k)^2}{2} +l$. Thus replacing $y$ by $y+k$, we have
$$
X\,=\, e^{x} \partial_{y},\ Y \,=\, e^{-x}\left(y\partial_{x}+ \left(\frac{y^2}{2} +l\right)\right)\partial_{y}.
$$

\textbf{Case IIc:} \textit{ Assume that $b_{z}  \,\neq\, 0.$}\, We have
$X\,=\, e^{x} \partial_y$,\, $Y\,=\, e^{-x} (a(y,z) \partial_x + b(y,z) \partial_y)$.
So $H\,=\, \partial_x\,=\,[X,\,Y]$ gives
$$
    \frac{\partial a}{\partial y}\,\,=\,\,1,\ \ \, \frac{\partial b}{\partial y}\,\,=\,\,a.
$$
Hence, $a\,=\,  y+\phi(z)$,\, $b\,=\, (y+\phi(z))^2/2+\psi(z)$. The transformation
$\widetilde{x}\,=\,x$,\, $\widetilde{y} \,=\, y+ \phi(z)$,\, $\widetilde{z}\,=\,z$ gives
$$
X \,\,=\,\, e^{\widetilde{x}} \partial_{\widetilde{y}},\ \
Y \,\,=\,\, e^{-\widetilde{x}}\left(\widetilde{y} \partial_{\widetilde{x}} + \left ( \frac{\widetilde{y}^2}{2}+ \psi(z)\right) \partial_{\widetilde{y}}\right),
$$
where $\psi'(z) \,\neq\, 0$. By a further change of variables
$\widetilde{x}\,=\,x$,\, $\widetilde{y}\,=\, y$,\, $\widetilde{z}\,=\,\psi(z)$, we have after removing bars,
$$
 X \,\,=\,\, e^{{x}} \partial_{{y}},\ \ \,
Y \,=\, e^{-{x}}\left({y} \partial_{{x}} + \left ( \frac{{y}^2}{2}+ z\right) \partial_{{y}}\right).
$$

\textbf{Case III.} \textit{ Assume that the rank of $\mathfrak g$ is 3}. As before
\begin{align}
&X\,\,=\,\,e^{x}(a(y,z)\partial_{x}+b(y,z)\partial_{y}+c(y,z)\partial_{z}), \nonumber \\
&Y\,\,=\,\, e^{-x}(a_1(y,z)\partial_{x}+b_1(y,z)\partial_{y}+c_{1}(y,z), \partial_{z}) ,\nonumber
\end{align}
and rank of one of $\langle X,\, H\rangle $ or $\langle Y,\, H \rangle $ must be 2, say rank of
algebra $\langle X,\, H\rangle $ is 2. Thus as rank of the commuting vector fields
$\partial_{x}$,\, $(a(y,z)\partial_{x}+b(y,z)\partial_{y}+c(y,z)\partial_{z})$ is 2, there is
a change of variables in which $H\,=\, \partial_{x}\,=\, \partial_{\widetilde{x}}$ and
$(a(y,z)\partial_{x}+b(y,z)\partial_{y}+c(y,z)\partial_{z}) \,=\, \partial_{\widetilde{y}}$.

Following the same argument as in Case II, we may assume that $H\,=\,\partial_{x}$,\, $X\,=\,
e^{x}\partial_{y}$,\, $Y \,=\, e^{-x} (a(y,z)\partial_{x}+b(y,z)\partial_{y}+c(y,z)\partial_{z})$
with $c\,\neq\, 0,$ because of the assumption that rank of $\mathfrak g$ is 3. Hence,
$[X,\,Y]\,=\,H\,=\,\partial_{x}$ implies
$$
    -a + \frac{\partial b}{\partial y}\,=\,0, \ \ \,
\frac{\partial a}{\partial y}=1,\frac{\partial c}{\partial y}\,=\,0.
$$
Therefore, $a\,=\, y + \tau(z)$,\, $b\,=\,y^2/2+\tau(z)y+\mu(z)$ and $c\,=\,c(z)$. As
$c(z)\,\neq\, 0$, we can introduce coordinates
$$
\widetilde{x} \,=\, x,\ \widetilde{y}\,=\,y,\ \partial_{\widetilde{z}}\,=\, c(z) \partial_{z}.
$$
So we may assume, after changing notation, that
\begin{equation}
X \,=\, e^{{x}} \partial_{{y}},\ Y\,=\, e^{-{x}} \left((y+\tau(z)) \partial_{{x}} +\left(\frac{y^2}{2}+\tau(z)y+\mu(z)\right) \partial_{y}+ \partial_{z}\right).
\end{equation}
This can be transformed into the simpler form
\begin{equation}
X \,\,=\,\, e^{\widetilde{x}} \partial_{\widetilde{y}},\ \ \,
Y\,=\, e^{-\widetilde{x}} \left(\widetilde{y} \partial_{\widetilde{x}} +
\frac{\widetilde{y}^2}{2} \partial_{\widetilde{y}}+ \partial_{\widetilde{z}}\right),
\end{equation}
by using the point transformation $\widetilde{x}\,=\,x+ \ln{\alpha(z)}$,\,
$\widetilde{y}\,= \,\alpha(z)y+ \beta(z)$,\, $\widetilde{z}\,=\,h(z)$ where $\alpha,\, \beta,
\, h$ satisfy
$$
\beta \,=\, \alpha \tau + \alpha',\ \frac{\beta^2}{2}\,=\, \alpha^2 \mu + \alpha \beta',\
h'\,=\, \frac{1}{\alpha}.
$$
Summarizing all the above cases, we arrive at the following result.

\begin{prop}\label{prop2}
    Every (local) representation of $sl(2,\mathbb{C})$ in $V(\mathbb{C}^3)$ is equivalent to one of the following forms:
    \begin{enumerate}
        \item $\langle e^{x}\partial_{x},\,\, e^{-x}\partial_{x} \rangle$
        \item $\langle e^{x} \partial_{y},\,\, e^{-x} (y\partial_{x} + (y^2/2 +l) \partial_{z}) \rangle$
        \item $\langle e^{x} \partial_{y},\,\, e^{-x} (y\partial_{x} + (y^2/2 +z) \partial_{z}) \rangle $
        \item $\langle e^{x} \partial_{y},\,\,
e^{-x} (y\partial_{x} + y^2/2 \partial_{y} + \partial_{z}) \rangle $.
    \end{enumerate}
\end{prop}

\section{Canonical Forms of $sl(2,\mathbb{C})\times sl(2,\mathbb{C})$ in $V(\mathbb{C}^3)$}

Choose generators $X,\, X_{-},\, Y,\, Y_{-}$ so that if $H\,=\,[X,\,X_{-}]$,\,
$\widetilde{H}\,=\, [Y,\, Y_{-}]$ then
$$
    [H,\,X]\,=\,H,\ [H,\,X_{-}]\,=\, -X,\ [\widetilde{H},\,Y]\,=\,Y,\ [\widetilde{H},\,
Y_{-}]\,=\,-Y,
$$
and $\langle X,\, X_{-} \rangle $ commutes with $\langle Y,\, Y_{-} \rangle $. Choose coordinates
$x,\,y,\,z$ such that $H\,=\, \partial_{x},\, \widetilde{H}\,=\, \partial_{y}$. Then
$$
X\,=\, e^{x} U,\ X_{-}\,=\,e^{-x}V,\ Y\,=\,e^{y}W,\ Y_{-}\,=\, e^{-y}\widetilde{X},
$$
where the coefficients of $U,\,V,\,W,\, \widetilde{X}$ with respect to the basis
$\partial_{x},\, \partial_{y},\, \partial_{z}$ are functions of $z$.

Assume that $\mathfrak g$ is of rank 3. Thus one of the abelian algebras
$\langle \partial_{x}, \partial_{y}, U\rangle $,\, $\langle \partial_{x},\, \partial_{y},
\,W\rangle $,\, $\langle \partial_{x},\, \partial_{y},\, \widetilde{X}\rangle $ must be
of rank 3. Without loss of generality we may assume that $\langle \partial_{x},\,
\partial_{y},\, U\rangle $ is of rank 3. Consequently,
$U\,=\,a(z)\partial_{x}+b(z)\partial_{y}+c(z)\partial_{z}$ with $c(z) \,\neq\, 0$.
Hence there is a change of variables $\widetilde{x},\, \widetilde{y} ,\, \widetilde{z}$, in
which $\partial_{x} \,=\, \partial_{\widetilde{x}}$,\, $\partial_{y} \,=\, \partial_{\widetilde{y}}$,
\, $\partial_{z} \,= \,\partial_{\widetilde{z}} $. Therefore,
$$
\widetilde{x} \,=\, x+\phi(z),\ \widetilde{y} \,=\, y+\psi(z) ,\ \widetilde{z} \,=\, F(x,y,z),
$$
and $c(z)F'\,=\,1$. Hence,
$X\,=\,e^{x} \partial_{\widetilde{z}} \,=\, e^{\widetilde{x}+ \phi(z)} \partial_{\widetilde{z}}
\,=\, e^{\widetilde{x}} e^{\widetilde{\phi}(\widetilde{z})}\partial_{\widetilde{z}}$. We can
write $e^{\widetilde{\phi}(\widetilde{z})}\partial_{\widetilde{z}}\,=\, \partial_{\dbtilde{z}}.$
Therefore, by a change of notation we may assume that
$$
X \,=\, e^{{x}} \partial_{{z}},\ \,Y\,=\, e^{-{x}} \left(a_{1}(z)\partial_{{x}} +a_{2}(z)\partial_{y}+ a_{3}(z)\partial_{z}\right).
$$
Using $[X,\,X_{-}]\,=\,\partial_{x}$, we see that $a_{1}'(z)\,=\,1$,\, $b_{1}'(z)\,=\,0$,\,
$c_{1}'(z)\,=\,a(z)$. Hence, $a(z)\,=\,z+l_{1},$\, $b(z)\,=\,l_{2}$ and $c(z)
\,=\,(z+l_1)^2/2+l_{3}.$ If we put $\widetilde{z}\,=\,z+l_{1}$,\, $\widetilde{x}\,=\,y$,
\, $\widetilde{z}\,=\,z$, we may assume that
$$
X \,\,=\,\, e^{{x}} \partial_{{z}},\ \ \, Y\,\,=\,
\,e^{-{x}} \left(z\partial_{{x}} +l_{2}\partial_{y}+ (\frac{z^2}{2}+l_{3})\partial_{z}\right).
$$
Thus we may assume that
$$
X_{-} \,\,=\,\, e^{-{x}} \left(z\partial_{{x}} +\alpha \partial_{y}+
\left (\frac{z^2}{2}+\beta \right )\partial_{z}\right),
$$
where $\alpha$ and $\beta $ are constants. Now $Y\,=\,e^{y} (h_{1}(z)\partial_{x}+
h_{2}(z)\partial_{y}+ h_{3}(z)\partial_{z})$ and it commutes with $X$ and $X_{-}$. Therefore,
$$
h_{1}'\,=\,0, \ h_{2}'\,=\,0,\ h_{3}'\,=\, h_{1}.
$$
Hence, $h_{1}\,=\,c_{1},\, h_{2}\,=\,c_{2},\, h_{3} \,=\, c_{1}z +c_{3}$. Therefore, $Y
\,=\,e^{y} (c_1\partial_{x}+c_2\partial_{y}+ (c_{1}z+c_{3}))\partial_{z})$. One can check that
$c_1$ and $c_2$ both can't be zero. Therefore, if $c_1\,=\,0$, then $c_2 \,\neq\, 0$
and $Y\,=\,e^{y}(\partial_{y}+ \lambda \partial_{z})$ (up to a scalar). We now have
$X \,=\, e^{{x}} \partial_{{z}}$,\, $Y\,=\, e^{y} \left(\partial_{y}+
\lambda \partial_{z}\right)$,\, $X_{-} \,=\,
e^{-{x}} \left(z\partial_{{x}} +\alpha \partial_{y}+ \left (z^2/2+\beta \right )
\partial_{z}\right)$ and $[Y,\,X_{-}]\,=\,0$ gives $\lambda \,=\,0,\, \alpha \,=\,0.$
Consequently,
$$
X \,=\, e^{{x}} \partial_{{z}},Y= e^{y} \partial_{y},\ \, X_{-}
\,=\, e^{-{x}} \left(z\partial_{{x}} + \left (\frac{z^2}{2}+\beta \right )\partial_{z}\right).
$$
Now $[X,\,Y_{-}]\,=\,0$ and $[Y,\, Y_{-}]$ is a non-zero multiple of $\partial_{y}$ gives
$Y_{-}\,=\, e^{-y}\partial_{y}.$ This gives one form of $sl(2,\mathbb{C})\times sl(2,\mathbb{C})$:
$$
X\, =\, e^{{x}} \partial_{{z}},\, Y\,=\, e^{{y}} \partial_{y},\, X_{-}
\,=\, e^{-{x}} \left(z\partial_{{x}} + \left (\frac{z^2}{2}+\beta \right )\partial_{z}\right),
\, Y_{-} \,= \,e^{-y} \partial_y.
$$
Assuming $c_{1}\,\neq\, 0$, we may assume that $Y\,=\, e^{{y}}( \partial_{x} + l_{1}
\partial_{y} + (z+l_{2})\partial_{z})$ and $X\,=\,e^{x}\partial_{z}$,\,
 $X_{-}\,=\, e^{-{x}} \left(z\partial_{{x}} + a\partial_{y}+ (z^2/2 +b) \partial_{z}\right)$.
Using $[Y,\,X_{-}]\,=\,0$, we must have $a\,=\,l_{2}$ and $a+al_{1}\,=\,0$ and
$a^2 +2b\,=\,0.$ Considering the cases $a\,=\,0$, $a\,\neq\, 0$ and using the commutation relations, we get
the following:
\begin{prop}\label{prop3}
The possible forms of all algebras of type $sl(2,\mathbb{C})\times sl(2,\mathbb{C})$ in $V(\mathbb{C}^3)$ are
\begin{align}
1.\quad &X = e^{{x}} \partial_{{z}},\, Y= e^{{y}} \partial_{y},\, X_{-} =
e^{-{x}} \left(z\partial_{{x}} + \left (\frac{z^2}{2}+\beta \right )\partial_{z}\right),\,
Y_{-} = e^{-y} \partial_y
\nonumber \\
2. \quad &X = e^{{x}} \partial_{{z}},\, Y= e^{{y}} (\partial_{x}-\partial_{y}+z\partial_{z}),\,
 X_{-} = e^{-{x}} \left(z\partial_{{x}} +  \frac{z^2}{2}\partial_{z}\right),
\nonumber \\
&Y_{-}= e^{-y}(\partial_{x} + \partial_{y}+z\partial_{z})
\nonumber \\
3. \quad &X = e^{{x}} \partial_{{z}},\,Y= e^{{y}} (\partial_{x}-\partial_{y}+(z+a)\partial_{z}),\,
X_{-} = e^{-{x}} \left(z\partial_{{x}} + a\partial_{y}+  \left(\frac{z^2}{2}+b\right)\partial_{z}
\right),\nonumber\\
&Y_{-}= e^{-y}(\partial_{x} + \partial_{y}+(z-a)\partial_{z}), \ a^{2} + 2b = 0
\nonumber \\
4. \quad & \mbox{If the rank of $\mathfrak g$ is 2, then g = }~~ \langle e^x \partial_x,\,
e^{-x} \partial_x \rangle \times \langle e^{y} \partial_y,\, e^{-y} \partial_y\rangle. \nonumber
\end{align}
\end{prop}

\section{Simple Algebras of rank 2 in $V(\mathbb{C}^{3})$}

In order to classify all semisimple Lie algebras of vector fields in $\mathbb{C}^{3}$, it
remains, in view of Propositions 2 and 3, to classify forms of types $A_{2},B_{2}$ and $G_{2}$ in
$V(\mathbb{C}^{3})$.

We do this by extending forms of the Heisenberg algebra $H \subset A_{2}$, $A_{1}\times A_{1}
\,\subset\, B_{2} $ and $A_{2}\,\subset\, G_{2}$. For this we need to make a choice of structure
constants that are given by a model of these types of algebras or using the following idea.

\textbf{Algorithm:} Let $R$ be a root system.
\begin{enumerate}
    \item Fix a simple system $S$ of roots, say, $S\,=\,\{ \alpha_{1},\, \alpha_{2},
\, \cdots, \,\alpha_{N}\}$.

\item For a positive root $r$ which is not simple, let $\alpha_{i}$ be the first simple root so that $r-\alpha_{i}$ is a positive root.

\item Define $X_{r}$ inductively by $X_{r} \,=\, [ X_{\alpha_{i}},\, X_{r-\alpha_{i}}]$ and
$X_{-r} \,=\, [ X_{-\alpha_{i}},\, X_{-r+\alpha_{i}}]$.

\item  If $\alpha$ is a simple root and we define $H_{\alpha} \,=\, [X_{\alpha},\, X_{-\alpha}]$
then $[H_{\alpha},\, X_{\beta}] \,=\, \langle \beta,\, \alpha \rangle X_{\beta}$, where
$\langle \beta,\, \alpha \rangle$ is the Cartan integer for the pair $(\beta, \alpha)$.

\item All the structure constants $[X_{r},\, X_{s}] \,=\, N_{r,s} X_{r+s}$ are determined for all roots.
\end{enumerate}
We omit the proof as it reduces to rank 2 computations.

\subsection{Forms of $A_{2}$ in $V(\mathbb{C}^{3})$}

\begin{prop}\label{prop4}
If $\alpha$ and $\beta$ are the simple roots an algebra of type $A_{2}$  then every embedding of a Lie algebra of type $A_{2}$ in $V(\mathbb{C}^{3})$ is equivalent by a local change of variables to one of the following forms:
\begin{enumerate}
\item $X_{\alpha}\, =\, \partial_y,\, X_{\beta} \,=\, y \partial_{x},\, X_{-\alpha}
\,=\, -xy\partial_{x} -y^2 \partial_{y},\, X_{-\beta} \,=\, x \partial_{y}$

\item $X_{\alpha} \,=\, \partial_y,\, X_{\beta} \,=\, y \partial_{x},\, X_{-\alpha}
\,=\, -xy\partial_{x} -y^2 \partial_{y}+y\partial_{z}, X_{-\beta} \,=\, x \partial_{y}$

\item $X_{\alpha} \,=\, \partial_y,\, X_{\beta}
\,=\, y \partial_{x}+\partial_{z},\, X_{-\alpha} \,=\, -xy\partial_{x} -y^2 \partial_{y}+(yz-x)
\partial_{z},\, X_{-\beta} \,=\, x \partial_{y}-z^2 \partial_{z}$.
\end{enumerate}
\end{prop}

\begin{proof}
The idea of the proof is to extend the forms of the algebra generated by $X_{\alpha}$ and
$X_{\beta}$ - namely $X_{\alpha},\, X_{\beta},\, [X_{\alpha},\,X_{\beta}]\,
=\, X_{\alpha+\beta}$ to the algebra $\langle X_{\alpha},\, X_{\beta},\, X_{-\alpha},\, X_{-\beta}\rangle$. We give details of only Type (3) as the other forms follow using similar steps.

Type (3) of the algebra $\langle X_{\alpha} ,\, X_{\beta} \rangle $ is
$$
X_{\alpha} \,=\, \partial_{y} ,\, X_{\beta} \,=\, y \partial_{x} + \partial_{z},\,
X_{\alpha +\beta} \,=\, \partial_{x}.
$$
Using $[X_{-\beta},\, X_{\alpha} ] \,=\,0$ and $[ X_{\alpha +\beta},\, X_{-\beta}]
\,=\, X_{\alpha}$, we see that
$$
X_{-\beta} \,=\, f_{1}(z) \partial_{x}+ x\partial_{y}+f_{3}(z) \partial_{z}.
$$
With $[X_{\beta},\, X_{-\beta}]\,=\,H_{\beta}$,\, $[H_{\beta},\, X_{\alpha} ]
\,=\, - X_{\alpha}$ and $[X_{\beta}, \,H_{\beta}]\,=\,-2X_{\beta}$, we see that $f_{1}''
\,=\,0$,\, $f_{3}''\,=\,-2$. So $f_{1}\,=\, az+b$,\, $f_{3} \,= \,-z^2 +cz+d$. Hence,
\begin{equation}
X_{-\beta} \,\,=\,\, (az+b) \partial_{x} + x \partial_{y} - \left ( \left (z-\frac{c}{2} \right)^2
 +d - \frac{c^2}{4}\right) \partial_{z}.
\end{equation}
If we put $\widetilde{z} \,=\, z-c/2$, we may assume, after a change of notation, that
$$
X_{-\beta} \,=\, (az+b) \partial_{x} + x \partial_{y} - \left (  z ^2 +k\right) \partial_{z}.
$$
Hence, $H_{\beta} \,=\, y\partial_{y} - (x-a) \partial_{x} -2z \partial_{z}$. Now $X_{-\alpha}
\,=\, F_{1} \partial_{x} + F_{2} \partial_{y} +F_{3} \partial_{z}$. Using $[X_{\alpha+\beta} ,
\, X_{-\alpha}]\,=\,-X_{\beta}$ and $[X_{-\alpha},\, X_{\beta} ] \,=\,0$, we see that
$$
X_{-\alpha} \,\,=\,\, \left ( -xy +(G_{2}(y) + y^2)z +  \phi(y) \right ) \partial_{x} + G_{2}(y) \partial_{y} + (-x+yz+G_{3}(y)) \partial_{z}.
$$
As $X_{\alpha} \,=\, \partial_{x}$, and $H_{\alpha} \,=\, [X_{\alpha} ,\, X_{-\alpha}]$
and $[X_{\alpha},\, H_{\alpha}]\,=\, -2X_{\alpha} \,=\, -2 \partial_{y}$, we see that $G_{2}''
\,=\,-2$,\, $G_{3}''\,=\,0$ and $\phi''\,=\,0$. Finally using that $[X_{-\alpha} ,\, X_{-\beta}]
\,=\, X_{-\alpha -\beta}$ commutes with $X_{-\alpha}$ and $X_{-\beta}$ and $[[X_{r},\,
X_{-r}],\, X_{s}]\,=\, \langle s,\,r \rangle X_{s}$, where $r \,=\, \alpha,\, \beta$ and
$s\,=\,\pm \alpha,\, \pm \beta$, we get a system of equations whose solution gives
$$
X_{-\alpha} \,=\,  -xy \partial_{x} -y^2 \partial_{y} + (gy+yz-x) \partial_{z},\ \,
 X_{-\beta} \,=\,  b \partial_{x} +x\partial_{y} - (z+g)^2 \partial_{z}.
$$
Using $[X_{-\alpha-\beta},\, X_{-\beta}]\,=\, 0$, we find that $b\,=\,0$. Finally putting
$\widetilde{z} \,=\, z+g$ and changing notation we get
$$
X_{-\alpha} \,=\,  -xy \partial_{x} -y^2 \partial_{y} + (yz-x) \partial_{z},\ \,
 X_{-\beta} \,=\,  x\partial_{y} - z^2 \partial_{z}.
$$
This completes the proof.
\end{proof}

\subsection{Forms of $B_{2}$ in $V(\mathbb{C}^3)$}

\begin{prop}\label{prop5}
If $\alpha$ and $\beta$ are the simple roots of an algebra of type $B_{2}$ with $\alpha$ a long root, then every embedding of an algebra of type $B_{2}$ in $V(\mathbb{C}^3)$ is equivalent by a local change of
coordinates to one of the following two:
\begin{align}
     &X_{\alpha}\,=\, e^{x}(\partial_x+\partial_y+z\partial_z) , X_{-\alpha}= e^{-x}(-\partial_x+\partial_y+z\partial_z) \nonumber \\
      &X_{\beta}\,=\, e^{\frac{-x+y}{2}}\left(\partial_x-\partial_y- \left(z+\frac{1}{4}\right)\partial_z\right) \nonumber\\
      &X_{-\beta} \,=\, e^{\frac{x-y}{2}} \left(z\partial_x+\left(z+\frac{1}{2} \right )\partial_y+\left(z^2+\frac{z}{4}\right)\partial_z\right) \nonumber
\\
& \quad \mbox{or} \nonumber
\\
&
X_{\alpha}\,=\, e^{x}(-\partial_x+\partial_y+(z+a)\partial_z) \nonumber \\
     &X_{-\alpha}\,=\, e^{-x}(\partial_x+\partial_y+(z-a)\partial_z) , X_{\beta}= e^{\frac{-x+y}{2}}(\partial_y+ (z-a)\partial_z) \nonumber\\
      &X_{-\beta} \,=\, e^{\frac{x-y}{2}}\left(\partial_x+\left(\frac{z+a}{2a}-1 \right)\partial_y+ \left(\frac{(z+a)^2}{2a} - (z+a) \right) \partial_z \right) \nonumber
\end{align}
\end{prop}

\begin{proof}
The algebra $\langle X_{\alpha},\, X_{-\alpha} \rangle \times \langle X_{\alpha +2\beta},\,
X_{-\alpha-2\beta} \rangle $ is of type $A_{1}\times A_{1}.$ We will use the forms of $A_{1}\times A_{1}$
given in Section 5 to extend them to $B_2$. We use the following model of $B_2$ to fix the structure
constants. First of all, $B_2 \,\cong\, C_{2}, $ is a subalgebra of $sl(4,\mathbb{C})$. The root vectors for the
simple roots and their negatives are
$X_a \,=\, y \partial_z$,\, $X_b \,=\, x \partial_y +z \partial_w$,\, $X_{-a}\,=\, z \partial_y$,\,
$X_{-b} \,=\, y \partial_x + w \partial_{z} $. Thus, $\langle X_{a},\, X_{-a}/2\rangle \times \langle X_{a+2b},\, X_{-a-2b}/8 \rangle$ is a standard copy of $sl(2, \mathbb{C}) \times sl(2, \mathbb{C})$.

We want to write down the entire algebra in the canonical coordinates $x,\,y$ of $[ X_{a},\, X_{-a}/2 ]$ and
$[ X_{a+2b},\, X_{-a-2b}/8 ]$. As $X_\beta$, $X_{-\beta}$ must be common eigenvectors of $\partial_{x}$ and $\partial_{y}$, these must be of the forms
\begin{align}
    & X_{\beta}\,=\,e^{ax+by}(f_{1}(z)\partial_x + f_{2}(z)\partial_y + f_{3}(z)\partial_z  )\\
    &  X_{-\beta}\,=\,e^{-ax-by}(g_{1}(z)\partial_x + g_{2}(z)\partial_y + g_{3}(z)\partial_z  ).
\end{align}
Using the above model of $B_2$, we see that $a\,=\,-1/2$ and $b\,=\,1/2$. Here, by definition,
\begin{align}
    &X_{\alpha +\beta} \,=\, [X_{\alpha},\, X_{\beta}] ,\ \, X_{\alpha +2\beta} \,=\, [X_{\beta},\, X_{\alpha+\beta}] \\
    &X_{-\alpha -\beta} \,=\, [X_{-\alpha},\, X_{-\beta}] ,\ \,
    X_{-\alpha-2\beta} \,=\, [X_{-\beta},\, X_{-\alpha-\beta}] .
\end{align}
Moreover if $r+s \,\neq\, 0$, and $r+s$ is not a root, the $[X_r,\, X_s]\,=\,0$. In particular, $[X_{\beta},\,
X_{-\alpha}] \,=\, 0$,\, $[X_{\beta},\, X_{\alpha + 2\beta }] \,=\, 0$ and $[X_{-\beta},\, X_{\alpha}]
\,=\, 0$,\, $[X_{-\beta},\, X_{-\alpha-2\beta}] \,=\, 0$. The forms of $A_1 \times A_1$ are
\begin{enumerate}
\item $\langle e^{x}\partial_{x},\, e^{-x}\partial_x \rangle \times \langle e^{y}\partial_{y},\, e^{-y}\partial_y \rangle $
\item $\langle e^{x}\partial_{z} ,\,e^{-x} z^2/2+b)\partial_{z}\rangle  \times \langle e^{y}\partial_{y},\, e^{-y}\partial_y \rangle$
\item   $\langle e^{x}\partial_{z},\, e^{-x}(  z\partial_{x}+z^2/2 \partial_{z})  \rangle  \times \langle e^{y}
( \partial_x +l \partial_y + z\partial_z) ,\, e^{-x}(\partial_x -l\partial_y+z\partial_z) \rangle$
\item   $\langle e^{x}\partial_{z} , \,e^{-x}(  z\partial_{x}+a\partial_y+(z^2/2+b) \partial_{z})  \rangle
\times \langle e^{y} ( \partial_x - \partial_y + (z+a)\partial_z) ,\, e^{-y}(\partial_x +\partial_y+(z-a)\partial_z) \rangle$
where $a^2 +2b\, =\,0.$
\end{enumerate}

Using the commutation relations, we see that forms (1) and (2) of $A_{1} \times A_{1}$ cannot be extended
to $B_2$ as one gets $X_{\beta}\,=\,0$ or $X_{\alpha+\beta} \,=\,0$.

Assume that the embedding of $\langle X_{\alpha} ,\, X_{-\alpha}\rangle \times \langle X_{\alpha+2\beta} ,\,
X_{-\alpha-2\beta}\rangle $ is of type (3). Now if $\alpha,\, \beta$ is a simple set of roots then
so is $\alpha+2\beta ,\,-\beta$. Therefore, we may assume that
\begin{align}
    & X_{\alpha} \,=\, e^{x}(l\partial_{x}+\partial_y + z\partial_z) ,\ \, X_{-\alpha}
\,=\, e^{-x}(-l\partial_{x}+\partial_y + z\partial_z), \nonumber \\
    & X_{\alpha+2\beta } \,=\, e^{y}\partial_z,\ \,  X_{-\alpha-2\beta} \,=\, e^{-y} (z\partial_y +z^2/2\partial_z) ,\nonumber \\
    & X_{\beta} \,=\, e^{\frac{-x+y}{2}} (f(z) \partial_x + g(z) \partial_y + h(z) \partial_z).
\end{align}
Using $[X_{\alpha} , \,X_{\alpha+\beta}] \,=\,0, $ we see that $f'\,=\,0\,=\,g'$ and $h'\,=\,g$. Thus,
$f\,=\,c_{1},\, g\,=\,c_{2},\, h\,=\,c_{2}z+c_{3}$. Using $[X_{\beta},\,X_{-\alpha}] \,=\,0$, we see that
\begin{align}
    &\frac{c_{1}(l-1)}{2}\,=\,0,\ \, c_{3}(-l+1) \,=\,0,\ \, c_{1}\,=\,\frac{-c_{2}(l+1)}{2}.
\end{align}
Now $l-1\,\neq\, 0$ gives $X_{\beta}\,=\, e^{\frac{-x+y}{2}} ( \partial_y +z\partial_z)$, $l\,=\,-1$ and
$X_{\alpha}\,=\,e^{x}(-\partial_x +\partial_y+z\partial_z)$. This implies that $X_{\alpha+2\beta}
\,=\,0$. Therefore, $l$ must be 1 and $X_{\alpha}\,=\,e^{x}(\partial_x +\partial_y+z\partial_z)$ and
$X_{\beta} \,=\,  e^{\frac{-x+y}{2}} (c_{1} \partial_x - c_1 \partial_y +(c_3 -c_1z) \partial_z) $.

If $c_1 \,=\,0,$ then $X_{\beta} \,=\,  e^{\frac{-x+y}{2}} \partial_z $ (up to a multiple). But then $X_{\alpha+2\beta}
\,=\,0. $ Therefore, $c_1\,\neq\, 0$ and we may take
$$
    X_{\beta} \,\,=\,\,  e^{\frac{-x+y}{2}}(\partial_x - \partial_y -(z+\lambda) \partial_z).
$$
Now,
$$
X_{-\beta}\,\,=\,\,  e^{\frac{x-y}{2}}(\widetilde{f}(z)\partial_x +\widetilde{g}(z) \partial_y +\tilde{h}(z) \partial_z).
$$
Using $[X_{\alpha},\, X_{-\beta} ]\,=\,0$, we get the system
$$
-\widetilde{f}+z \widetilde{f}'\,=\,0,\ -\widetilde{f}+z\widetilde{g}'\,=\,0,\ -\widetilde{f}z +
z\widetilde{h}'-\widetilde{h}\,=\,0.
$$
Solving this system, we get
$$
X_{-\beta}\,\, =\,\,  e^{\frac{x-y}{2}}(k_1 z\partial_x +(k_1 z+k_2)  \partial_y + (k_1 z^2 +k_3 z)  \partial_z).
$$
Keeping in mind that $[X_{\beta},\, X_{-\beta}]$ is a non-zero multiple of $\partial_x - \partial_y$,
we see that $k_1 \,\neq\, 0.$ Hence, we may assume that
$$
    X_{-\beta} \,=\,  e^{\frac{x-y}{2}}(z\partial_x +(z+\mu)  \partial_y + (z^2 +\tau z)  \partial_z).
$$
As $[X_{-\beta},\, X_{\alpha+\beta}]$ is a non-zero multiple of $X_{\alpha} $, we see that $\mu +2 \lambda \,
=\,1$,\, $2\tau \,= \,\mu$ and $\lambda \,=\, 1/2-\tau$. Thus,
$$
X_{-\beta} \,=\,  e^{\frac{x-y}{2}}(z\partial_x +(z+2\tau)  \partial_y + (z^2 +\tau z)  \partial_z).
$$
Now using that $[X_{\beta},\, X_{-\beta}]$ is a non-zero multiple of  $\partial_x- \partial_y$
it is deduced that $4\tau -1\,=\,0$. Thus, $\tau = 1/4$, $\lambda =1/4$ and $\mu = 1/2$. Therefore,
\begin{align}
    &X_{-\beta} \,=\,  e^{\frac{x-y}{2}}\left (z\partial_x +\left(z+\frac{1}{2}\right)  \partial_y + \left (z^2 +\frac{z}{4} \right)  \partial_z \right),\nonumber \\
    &X_{\beta}\,=\, e^{\frac{-x+y}{2}  } \left(\partial_x - \partial_y - \left ( z+ \frac{1}{4} \right) \right),\nonumber \\
    & X_{\alpha}\,=\, e^{x} (\partial_x + \partial_y + z\partial_z),
    X_{-\alpha}= e^{-x} (-\partial_x + \partial_y + z\partial_z).\nonumber
\end{align}

Using similar computations, we see that if the embedding of
$\langle X_{\alpha} ,\, X_{-\alpha}\rangle \times \langle X_{\alpha+2\beta} ,\, X_{-\alpha-2\beta}\rangle $ is of
type (4); then
\begin{align}
    &X_{\beta} \,=\,  e^{\frac{-x+y}{2}}(\partial_y +(z-a)\partial_z), \nonumber \\
    &X_{-\beta} \,=\,  e^{\frac{x-y}{2}}\left (\partial_x +\left(\frac{z+a}{2a}-1\right)  \partial_y + \left (\frac{(z+a)^2}{2a} - (z+a)\right)  \partial_z \right).
\end{align}
This completes the proof of the proposition.
\end{proof}

\section{Non-embedding of $G_2$ in $V(\mathbb{C}^3) $}

In this section, we want to show that there is no embedding of the Lie algebra $G_2$ in
$V(\mathbb{C}^3) $. If the simple roots of $G_2$ are $\alpha$ and $\beta$ with $\alpha$ the long root, then the positive roots are $\alpha$,\, $\beta$,\, $\alpha + \beta$,\,
$\alpha+2\beta$,\, $\alpha+3\beta$\, $2\alpha+3\beta.$

The long roots $\alpha,\, \alpha+3\beta,\, 2\alpha+3\beta$ and their negatives gives a copy of $A_2$. Assume
that there is a representation of $G_2$ in $V(\mathbb{C}^3) $. We see that $\langle X_{\alpha},\,
X_{-\alpha},\, X_{\alpha+3\beta} ,\, X_{-\alpha-3\beta}\rangle $ gives a copy of $A_2$ in $V(\mathbb{C}^3) $. The three forms of $A_2$ in $V(\mathbb{C}^3) $ are
\begin{enumerate}
\item $X_{a} \,=\, \partial_y,\, X_b \,= \,y\partial_x + \partial_z ,\,
X_{-a} \,=\,-xy\partial_x -y^2 \partial_y+(yz-x) \partial_z ,\, X_{-b} \,=\, x\partial_y-z^2 \partial_z$
\item $X_{a} \,=\, \partial_y,\ X_b \,=\, y\partial_x ,\ X_{-a} \,=\,-xy\partial_x -y^2 \partial_y,\
 X_{-b} \,=\, x\partial_y$
\item $X_{a} \,=\, \partial_y,\, X_b \,=\, y\partial_x ,\, X_{-a} \,=\,-xy\partial_x -y^2 \partial_y+y\partial_z,\,
X_{-b} \,=\, x\partial_y$.
\end{enumerate}
Using these three forms, we find that
\begin{equation}
    X_{\alpha +\beta} \,\,=\,\, f(z) \partial_x + g(z) \partial_y + h(z) \partial_z.
\end{equation}
As $[X_{-\alpha},\, X_{\alpha+\beta} ]\,=\,N_{-\alpha, \alpha+\beta}X_{\beta}$, we can take that $X_{\beta} \,=
\, [X_{-\alpha},\, X_{\alpha+\beta}]$. In every case, we find that $X_{\alpha+2\beta} \,=\,0$ or $X_{2\alpha+3\beta}
\,=\,0$ ---- a contradiction. Thus, there is no representation of $G_2$ in $V(\mathbb{C}^3).$

\end{document}